\documentclass[times]{elsarticle}

\journal{Journal of Mathematical Analysis and Applications}

\usepackage{amsmath,amssymb}
\usepackage{mathrsfs}
\usepackage[T1]{fontenc}
\usepackage[cp1250]{inputenc}

\newtheorem{thm}{Theorem}[section]
\newtheorem{cor}[thm]{Corollary}
\newtheorem{lem}[thm]{Lemma}
\newtheorem{prop}[thm]{Proposition}

\newtheorem{rem}[thm]{Remark}
\newdefinition{exmp}[thm]{Example}

\newproof{proof}{Proof}
\newproof{proof_main}{The final step of the proof of Theorem~\ref{th:main}}
\newproof{proof7}{Proof of the formula~\eqref{lem7}}
\newproof{proof8}{Proof of the formula~\eqref{lem8}}
\newproof{proof10}{Proof of the formula~\eqref{lem10}}

\newcommand{\N}{\mathbb{N}}
\newcommand{\R}{\mathbb{R}}
\newcommand{\I}{\mathcal{I}}
\newcommand{\mr}{\mathscr{M}(\R)}
\newcommand{\br}{\mathscr{B}(\R)}

\newcommand{\st}{such that}
\newcommand{\cf}{cf.}
\newcommand{\eg}{e.g.}
\newcommand{\ie}{i.e.}

\newcommand{\itemref}[1]{\eqref{#1}} 
\renewcommand{\le}{\leqslant}
\renewcommand{\ge}{\geqslant}

\newcommand{\J}{\mathcal{J}}

\numberwithin{equation}{section}

\begin{document}

\begin{frontmatter}
\author{Kazimierz Nikodem}\ead{knikodem@ath.bielsko.pl}
\author{Teresa Rajba}\ead{trajba@ath.bielsko.pl}
\author{Szymon Wąsowicz}\ead{swasowicz@ath.bielsko.pl}
\address{%
 Department of Mathematics and Computer Science,
 University of Bielsko--Bia\l{}a,
 Willowa 2, 43--309 Bielsko--Bia\l{}a, Poland}
\title{On the classes of higher--order Jensen--convex functions and Wright--convex functions}

\begin{abstract}
 The classes of $n$--Wright--convex functions and $n$--Jensen--convex functions are compared with each other. It is shown that for any odd natural number~$n$ the first one is the proper subclass of the second one. To reach this aim new tools connected with measure theory are developed.
\end{abstract}
\end{frontmatter}

\section{Introduction}

Let $\I\subset\R$ be the interval and $f:\I\to\R$. The usual \emph{forward difference} operator is denoted by
\[
 \Delta_h\,f(x)=f(x+h)-f(x),
\]
where $x\in\I$ and $h\in\R$ with $x+h\in\I$. Its iterates we define by the usual way, \ie
\[
 \Delta_{h_1\dots\,h_nh_{n+1}}\,f(x)=\Delta_{h_1\dots\,h_n}\bigl(\Delta_{h_{n+1}}\,f(x)\bigr)
\]
for $n\in\N$, $x\in\I$ and $h_1,\dots,h_n,h_{n+1}\in\R$ with all needed arguments belonging to~$\I$ (sometimes we will not write the evident assumptions of this kind). If all increments are equal, $h_1=\dots=h_n=h$, then we use the standard notation
\[
 \Delta_h^n\,f(x)=\Delta_{h\,\dots\,h}\,f(x)\,,
\]
where the increment~$h$~is taken~$n$ times. It is not difficult to check that
\begin{equation}\label{eq:delta_n}
 \begin{aligned}
  \Delta_{h_1\dots\,h_{n+1}}\,f(x)&=f(x+h_1+\dots+h_{n+1})\\
  &-\sum_{1\le j_1<\dots<j_n\le n+1} f(x+h_{j_1}+\dots+h_{j_n})\\
  &+\sum_{1\le j_1<\dots<j_{n-1}\le n+1} f(x+h_{j_1}+\dots+h_{j_{n-1}})\\
  &\phantom{\;}\,\vdots\\
  &+(-1)^n\sum_{1\le j_1\le n+1} f(x+h_{j_1})\\
  &+(-1)^{n+1}f(x)\,.
 \end{aligned}
\end{equation}
In this paper also the \emph{backward difference} will be used. It is defined by
\begin{equation}\label{eq:backward}
 \nabla_h\,f(x)=f(x)-f(x-h),
\end{equation}
where $x\in\I$ and $h\in\R$ with $x-h\in\I$. Its iterates are defined similarly to these of the forward differences. Obviously $\nabla_h\,f(x+h)=\Delta_h\,f(x)$ and using~\eqref{eq:delta_n}, by the induction argument we arrive at
\begin{equation}\label{eq:nabla_n}
 \nabla_{h_1\dots\,h_{n+1}}\,f(x+h_1+\dots+h_{n+1})=\Delta_{h_1\dots\,h_{n+1}}\,f(x)\,.
\end{equation}

Recall that $f$ is called \emph{Jensen--convex of order}~$n$ ($n$--Jensen--convex for short), if
\begin{equation}\label{eq:n_J_conv}
 \Delta_h^{n+1}f(x)\ge 0
\end{equation}
for all $x\in\I$ and $h>0$ with $x+nh\in\I$ (\cf\ \eg\ \cite{Kuc09}). Obviously for $n=1$ we arrive at the condition
\[
 \Delta_h^2\,f(x)=f(x+2h)-2f(x+h)+f(x)\ge 0
\]
for all $x\in\I$ and $h>0$ with $x+h\in\I$, which is equivalent to
\[
 f\biggl(\frac{x+y}{2}\biggr)\le\frac{f(x)+f(y)}{2},\quad x,y\in\I,
\]
\ie\ to the Jensen--convexity of~$f$.

The function $f$ is called \emph{Wright--convex} (\cf\ \cite{RobVar73}), if
\[
 f\bigl(tx+(1-t)y\bigr)+f\bigl((1-t)x+ty\bigr)\le f(x)+f(y)
\]
for all $x,y\in\I$, $t\in[0,1]$. This condition is equivalent to
\[
 \Delta_{h_1h_2}\,f(x)\ge 0
\]
for all $x\in\I$, $h_1,h_2>0$ with $x+h_1+h_2\in\I$ (see \cite{MakPal09}). Following this observation, in \cite{GilPal08} and \cite{MakPal09}, higher order Wright--convexity was defined: the function~$f$ is \emph{Wright--convex of order}~$n$ ($n$--Wright--convex for short), if
\begin{equation}\label{eq:n_W_conv}
 \Delta_{h_1\dots\,h_{n+1}}\,f(x)\ge 0
\end{equation}
for all $x\in\I$ and $h_1,\dots,h_{n+1}>0$ with $x+h_1+\dots+h_{n+1}\in\I$. Of course, setting above $h_1=\dots=h_{n+1}=h$, we obtain $\Delta_h^{n+1}f(x)\ge 0$, which means that every $n$--Wright convex function is $n$--Jensen convex.

Then the natural question arises, whether the converse is also true, \ie\ whether $n$--Jensen--convex functions are $n$--Wright--convex.  For $n=1$ the negative answer is not too difficult to give. Namely, the function $f:\R\to\R$ given by $f(x)=|a(x)|$, where $a:\R\to\R$ is a~discontinuous additive function, is Jensen--convex and it is not Wright--convex (\cf~\cite{Nik87}). Indeed, by the well--known Ng's representation (\cf~\cite{Ng87}), if~$f$ was Wright--convex, it would be the sum of an additive function and a~convex one. Then either $f$ would be continuous, or its graph would be dense on the whole plane (\cf~\eg~\cite{Kuc09}). But neither $f$ is continuous, nor~the graph of~$f$ is dense on the whole plane.

In the series of papers \cite{GilPal08,GilPal01,MakPal09} rather extensive study of higher--order Wright--convexity was given. However, the mentioned above problem was not considered. In this paper we fill this gap by delivering the negative answer for any odd positive integer~$n$. Let us emphasize that for (odd) $n>1$ the appropriate counterexample is not easy to construct, as it was for $n=1$, \ie\ in the case of the ordinary Jensen--convexity and Wright--convexity. To reach our goal we develop new tools of measure--theoretical nature, which, we hope, could be also useful for some future research. Let us also mention that for even natural numbers~$n$ the considered problem still remains open.

The paper is organized in the following way. In Section~\ref{sec:main} we formulate our main result and we prove a~part of it. In the next section, to throw some light to the nature of our main problem, we consider the case of $n$--Jensen--convexity and $n$--Wright--convexity for $n=3$. We also perform some considerations for $n=2$ to show that for even values of~$n$ our problem seems to be rather difficult. The nontrivial part of the proof of the main result is postponed to the last section.

\section{Main result}\label{sec:main}

Recall that for $x\in\R$~we have $x_+=\max\{x,0\}=\frac{x+|x|}{2}$ and $x_+^n=(x_+)^n$. We start with the following, well--known, lemma.
\begin{lem}\label{lem:cx+}
 Let $n\in\N$, $c\ge 0$. The function $\varphi:\R\to\R$ given by $\varphi(x)=cx_+^n$ is $n$--Jensen--convex.
\end{lem}
\begin{proof}
 It is easy to see that $\varphi^{(n-1)}(x)=cn!x_+$ is a~convex function, whence $\varphi$ is so--called $n$--convex function, which is obviously $n$--Jensen--convex (\cf~\cite{Kuc09,Pop34}).\qed
\end{proof}
\begin{cor}\label{cor:axn+}
 If $a:\R\to\R$ is an additive function and~$n$ is an odd natural number, then $f(x)=a(x)_+^n$ is $n$--Jensen--convex.
\end{cor}
\begin{proof}
 Let $\varphi(x)=x_+^n$. Then $f(x)=\varphi\bigl(a(x)\bigr)$. Using the well--known formula (\cf\ \cite[Corollary~15.1.2]{Kuc09}, see also~\eqref{eq:delta_n}) we obtain
 \begin{align*}
  \Delta_h^{n+1}f(x)&=\sum_{i=0}^{n+1}\binom{n+1}{i}(-1)^i f\bigl(x+(n+1-i)h\bigr)\\
  &=\sum_{i=0}^{n+1}\binom{n+1}{i}(-1)^i\varphi\bigl(a(x)+(n+1-i)a(h)\bigr)
  =\Delta_{a(h)}^{n+1}\,\varphi\bigl(a(x)\bigr)\ge 0\,,
 \end{align*}
 because $\varphi$ is an~$n$--convex function (for instance, by Lemma~\ref{lem:cx+}) and~$n$ is an odd number (if~$\varphi$ is~$n$--convex and~$n$~is odd, then $\Delta_k^{n+1}\varphi(y)\ge 0$ for any $y\in\R$ and any increment $k\in\R$, \cf\ \cite[p.~429]{Kuc09}, a~comment before Lemma~15.3.1). By virtue of~\eqref{eq:n_J_conv} the proof is finished.\hfill\null\qed
\end{proof}

In the rest of this paper we use the following idea. The additive map $a:\R\to\R$ is the linear functional over the vector space of real numbers over the field of rational numbers. Then the function~$a$ is uniquely determined by its values on the Hamel basis (\cf~\eg~\cite{Kuc09}).

Now we are in a~position to state our main result.
\begin{thm}\label{th:main}
Let $n$ be an odd natural number and let~$H\subset\R$ be the Hamel basis \st\ $h_1,\dots,h_{n+1}\in H$ are distinct and positive. Let $a:\R\to\R$ be the additive function \st
\[
 a(h_1)=-1,\quad a(h_2)=\dots=a(h_{n+1})=1.
\]
The function $f:\R\to\R$ given by
\begin{equation*}
  f(x)=\bigl(a(x)\bigr)_+^n
\end{equation*}
is $n$--Jensen--convex and it is not $n$--Wright--convex.
\end{thm}
\begin{proof}
 By Corollary~\ref{cor:axn+} the function~$f$ is $n$--Jensen convex. To prove that $f$ is not $n$--Wright--convex, it is enough to show that $\Delta_{h_1\dots\,h_{n+1}}\,f(0)=-1$ (see~\eqref{eq:n_W_conv}). However, this job is not trivial. It requires to develop new tools, and, on the other hand, it is rather long. For these reasons we postpone the rest of the proof to the last section.\hfill\null\qed
\end{proof}

Because every $n$--Wright convex function is $n$--Jensen convex, by the above Theorem we obtain immediately
\begin{cor}
 For any odd $n\in\N$ the class of $n$--Wright--convex functions is properly contained in the class of $n$--Jensen--convex functions.
\end{cor}
If $n\in\N$ is even, the question whether the above inclusion is proper, remains an open problem.

\section{Two particular cases}

\subsection{The case $n=3$}

As we mentioned in the Introduction, in the general case the proof of Theorem~\ref{th:main} is difficult. In this subsection we deliver some simpler proof for the case $n=3$.
\par\medskip
Take the Hamel basis $H$ \st\ $h_1,h_2,h_3,h_4\in H$ are distinct and positive. Let $a:\R\to\R$ be the additive function \st\ $a(h_1)=-1$, $a(h_2)=a(h_3)=a(h_4)=1$. Let $f(x)=\bigl(a(x)\bigr)_+^3$.
Due to Corollary~\ref{cor:axn+} the function $f:\R\to\R$ is 3--Jensen--convex. We will show that~$f$ is not $3$--Wright--convex. To this end we will check that $\Delta_{h_1h_2h_3h_4}\,f(0)=-1<0$, so the inequality~\eqref{eq:n_W_conv} does not hold for $n=3$. We have
\[
 f(0+h_1+h_2+h_3+h_4)=\bigl(a(0)+a(h_1)+a(h_2)+a(h_3)+a(h_4)\bigr)_+^3=8\,.
\]
Similarly
\begin{align*}
 f(0+h_1+h_2+h_3)=f(0+h_1+h_2+h_4)=f(0+h_1+h_3+h_4)&=1\,,\\
 f(0+h_2+h_3+h_4)&=27\,,\\[1ex]
 f(0+h_1+h_2)=f(0+h_1+h_3)=f(0+h_1+h_4)&=0\,,\\
 f(0+h_2+h_3)=f(0+h_2+h_4)=f(0+h_3+h_4)&=8\,,\\[1ex]
 f(0+h_1)&=0\,,\\
 f(0+h_2)=f(0+h_3)=f(0+h_4)&=1\,,\\[2ex]
 f(0)&=0\,.
\end{align*}
Then, having in mind the formula~\eqref{eq:delta_n}, we arrive at
\begin{align*}
 \Delta_{h_1h_2h_3h_4}^4f(x)&=f(0+h_1+h_2+h_3+h_4)\\[1ex]
 &-\bigl[f(0+h_1+h_2+h_3)+f(0+h_1+h_2+h_4)\\
 &\phantom{-\bigl[}+f(0+h_1+h_3+h_4)+f(0+h_2+h_3+h_4)\bigr]\\[1ex]
 &+\bigl[f(0+h_1+h_2)+f(0+h_1+h_3)+f(0+h_1+h_4)\\
 &\phantom{+\bigl[}+f(0+h_2+h_3)+f(0+h_2+h_4)+f(0+h_3+h_4)\bigr]\\[1ex]
 &-\bigl[f(0+h_1)+f(0+h_2)+f(0+h_3)+f(0+h_4)\bigr]\\[1ex]
 &+f(0)=8-30+24-3+0=-1\,.
\end{align*}
In a similar way this proof was also repeated for $n\in\{5,7,9,11\}$, however, the computations were done by the computer.

\subsection{The case $n=2$}

Now we discuss the case $n=2$ to convince the reader that for even values of~$n$ our problem is not easy to solve. Precisely, we will try to compare the classes of $2$--Jensen--convex functions with the class of $2$--Wright--convex ones.
\par\medskip
Looking at the example given in the Introduction we could suppose that the function $f:\R\to\R$ given by $f(x)=|Q(x)|$, where $Q:\R\to\R$ fulfils the quadratic functional equation
\begin{equation}\label{eq:quadratic}
 Q(x+y)+Q(x-y)=2Q(x)+2Q(y),
\end{equation}
could be a~good example of a~$2$--Jensen--convex function which is not $2$--Wright--convex. Unfortunately,~$f$ need not to be $2$--Jensen--convex. To see this take the Hamel basis~$H$ containing the vectors $1$, $\sqrt{2}$ and $\sqrt[4]{2}$. Next take the additive function $a:\R\to\R$ defined on~$H$ by $a(1)=-9$, $a(\sqrt{2})=4$ and $a(h)=0$ for $h\in H\setminus\bigl\{1,\sqrt{2}\bigr\}$. Then the function $Q(x)=a(x^2)$ fulfils~\eqref{eq:quadratic}. Finally, for $x=1$, $h=\sqrt[4]{2}-1>0$ we have $Q(x)=-9$, $Q(x+h)=4$, $Q(x+2h)=7$, $Q(x+3h)=0$, whence
\[
 \Delta_h^3f(x)=\Delta_h^3|Q(x)|=|Q(x+3h)|-3|Q(x+2h)|+3|Q(x+h)|-|Q(x)|=-18<0,
\]
which, according to~\eqref{eq:n_J_conv}, proves our claim.
\par\medskip
Having in mind Theorem~\ref{th:main}, it is reasonable to expect that the function $f(x)=\bigl(a(x)\bigr)_+^2$ (for some properly chosen additive function~$a:\R\to\R$) could be the nice example of a~$2$--Jensen--convex function which is not $2$--Wright--convex. However, such a~function is not $2$--Jensen--convex for any discontinuous additive function~$a$ and for any additive function of the form $a(x)=cx$ with $c<0$. If~$a(x)=cx$ with some $c\ge 0$, then~$f$ is continuous and, as we will show, $f$ is $2$--Jensen--convex. Hence, by continuity, $f$ is also $2$--Wright--convex (\cf~\cite[Theorem~15.7.1]{Kuc09}), so it is not a~good candidate for our counterexample.

\begin{prop}
 If $a:\R\to\R$ is a~discontinuous additive function, then $f(x)=\bigl(a(x)\bigr)_+^2$ is not $2$--Jensen--convex.
\end{prop}
\begin{proof}
 Since $a$ is~a discontinuous additive function, its graph is dense on the whole plane (\cf~\eg~\cite{Kuc09}). Then close to the point $(0,1)$ there exists a~point $\bigl(x,a(x)\bigr)$. We can claim, for example, that
 \begin{equation}\label{prop:2_conv_pf1}
  0.9<a(x)<1.1\,.
 \end{equation}
 Similarly, close to the point $(1,-2)$ there exists the point $\bigl(h,a(h)\bigr)$. We can claim that $h>0$ and
 \[
  -2.1<a(h)<-1.9\,.
  \]
 Therefore
 \begin{align*}
  -5.4<a(x+3h)=a(x)+3a(h)<-4.6&\implies \bigl(a(x+3h)\bigr)_+=0,\\
  -3.3<a(x+2h)=a(x)+2a(h)<-2.7&\implies \bigl(a(x+2h)\bigr)_+=0,\\
  -1.2<a(x+h)=(x)+a(h)<-0.8&\implies \bigl(a(x+h)\bigr)_+=0.
 \end{align*}
 By \eqref{prop:2_conv_pf1} we get $\bigl(a(x)\bigr)_+=a(x)>0$. Hence, by $f(x)=\bigl(a(x)\bigr)_+^2$,
 \[
  \Delta_h^3 f(x)=f(x+3h)-3f(x+2h)+3f(x+h)-f(x)=-\bigl(a(x)\bigr)^2<0,
 \]
 so the inequality~\eqref{eq:n_J_conv} does not hold for $n=2$ and for any $x\in\R$, $h>0$.\hfill\null\qed
\end{proof}

\begin{prop}
 If $a(x)=cx$ for some $c\ge 0$, then $f(x)=\bigl(a(x)\bigr)_+^2$ is $2$--Jensen--convex.
\end{prop}
\begin{proof}
 Since
 \[
  f(x)=\bigl(a(x)\bigr)_+^2=\biggl(\frac{cx+|cx|}{2}\biggr)^2=c^2\biggl(\frac{x+|x|}{2}\biggr)^2=c^2x_+^2\,,
 \]
 then $f$ is $2$--Jensen--convex by Lemma~\ref{lem:cx+}.\hfill\null\qed
\end{proof}

\begin{prop}
 If $a(x)=cx$ for some $c<0$, then $f(x)=\bigl(a(x)\bigr)_+^2$ is not $2$--Jensen--convex.
\end{prop}
\begin{proof}
 If $c<0$, then we have
 \[
  f(x)=\bigl(a(x)\bigr)_+^2=\biggl(\frac{cx+|cx|}{2}\biggr)^2=\biggl(\frac{cx-c|x|}{2}\biggr)^2=c^2\biggl(\frac{x-|x|}{2}\biggr)^2.
 \]
 Therefore
 \[
  f(-x)=c^2\biggl(\frac{-x-|x|}{2}\biggr)^2=c^2\biggl(\frac{x+|x|}{2}\biggr)^2=c^2x_+^2
 \]
 and $f(x)=c^2(-x)_+^2$. Setting $x=-1$, $h=1$ we obtain $\Delta_h^3f(x)=-c^2<0$, so $f$ is not $2$--Jensen--convex.\hfill\null\qed
\end{proof}

\section{Proof of Theorem~\ref{th:main}}

In this section we develop new tools connected with the measure theory and we use them to prove that the function~$f$ defined in Theorem~\ref{th:main} is not $n$--Wright--convex. According to our best knowledge this approach was not used so far.

\subsection{Notations and basic facts}

By $\br$ we denote the $\sigma$--field of Borel subsets of $\R$. By Borel measure we mean any measure defined on $\br$. It is known that the distribution function $F_{\mu}(x)=\mu\bigl((-\infty,x)\bigr)$ determines $\mu$ i.e.~to know the value the Borel measure, it is enough to know its values on the intervals $(-\infty,x)$ for any $x\in\R$ (\cf~\cite[Sections~12, 14]{Bil95}.

Throughout this section we deal only with the functions $f:\R\to\R$. In addition to the backward difference operator $\nabla_h$ given by~\eqref{eq:backward} we also consider the \emph{backward translation} operator
\[
 \tau_hf(x)=f(x-h)\,,\quad x,h\in\R\,.
\]
Let $\mr$ be the set of all Borel measures $\nu$ on $\br$ \st\ $\nu\bigl((-\infty,x)\bigr)<\infty$, $x\in\R$.

\begin{rem}\label{rem1}
 If $\nu\in\mr$, then $\lim\limits_{x\to-\infty}\nu\bigl((-\infty,x)\bigr)=0$.
\end{rem}
\begin{proof}
 It is an easy consequence of the general property of the measure: if $(A_k:k\in\N)$ is a~descending sequence of measurable sets with $\nu(A_1)<\infty$, then $\nu\bigl(\bigcap_{k\in\N}A_k\bigr)=\lim\limits_{k\to\infty}\nu(A_k)$.\hfill\null\qed
\end{proof}

We will consider the operators $\tau_h$ and $\nabla_h$ defined not only for the functions, but also for the measures $\nu\in\mr$ (such the approach is frequently used in the Measure Theory):
\[
 \tau_h\nu(B)=\nu(B-h)\,,\qquad \nabla_h\nu(B)=\nu(B)-\tau_h\,\mu(B)=\nu(B)-\nu(B-h)
\]
for $B\in\br$ with $\nu(B)<\infty$.
\par\medskip
Let $\nu\in\mr$. We define
\[
 \J_h\nu(B)=\sum_{n=0}^{\infty}\tau_h^n\nu(B)\,,\quad h>0\,,\;B\in\br\,,
\]
where $\tau_h^0\nu(B)=\nu(B)$, $\tau_h^{n+1}(B)=\tau_h\bigl(\tau_h^n\nu\bigr)(B)$. It is not difficult to check that
\[
 \J_h\nu\in\mr\iff \sum_{n=0}^{\infty}F_{\nu}(x-nh)<\infty\,,\;x\in\R\quad\text{and}\quad
 \lim_{x\to-\infty}\,\sum_{n=0}^{\infty}F_{\nu}(x-nh)=0\,.
\]
For these notations see also~\cite{Raj12b,Raj12a}.

\begin{prop}\label{prop2}
 Let $\mu\in\mr$ and $h>0$.
 \begin{enumerate}[\upshape 1.]
  \item\label{prop21}
   If
   \begin{equation}\label{prop2:2}
    \nabla_h\,\mu\ge 0\,,
   \end{equation}
   then there exists $\nu\in\mr$ \st\ $\mu$ has the form
   \begin{equation}\label{prop2:3}
    \mu=\J_h\nu\,.
   \end{equation}
   Moreover,
   \begin{equation}\label{prop2:4}
    \nu=\nabla_h\,\mu\,.
   \end{equation}
  \item\label{prop22}
   If $\mu$ has the form~\eqref{prop2:3} with $\nu\in\mr$, then the conditions~\eqref{prop2:2} and~\eqref{prop2:4} hold.
 \end{enumerate}
\end{prop}
\begin{proof}$ $

 \begin{enumerate}[\upshape 1.]
  \item
   Let $\mu$ fulfils~\eqref{prop2:2}. Using the definition of $\nabla_h$ we have $\nabla_h\,\mu=\mu-\tau_h\,\mu$, whence $\mu=\nabla_h\,\mu+\tau_h\,\mu$. Then
   \begin{align*}
    \tau_h\,\mu&=\tau_h\nabla_h\,\mu+\tau_h^2\,\mu\,,\\
    \tau_h^2\,\mu&=\tau_h^2\,\nabla_h\,\mu+\tau_h^3\,\mu\,,\\
    &\;\;\vdots\\
    \tau_h^n\,\mu&=\tau_h^n\,\nabla_h\,\mu+\tau_h^{n+1}\mu\,.
   \end{align*}
   Hence
   \begin{equation*}
    \mu=\nu+\tau_h\nu+\dots+\tau_h^n\nu+\tau_h^{n+1}\mu\,,
   \end{equation*}
   where $\nu=\nabla_h\,\mu$, $n=1,2,\dots$\,. Taking into account~Remark~\ref{rem1} we infer that
   \[
    \tau_h^{n+1}\mu\bigl((-\infty,x)\bigr)=\nu\Bigl(\bigl(-\infty,x-(n+1)h\bigr)\Bigr)\xrightarrow[n\to\infty]{}0\,.
   \]
   whence the distribution function of the measure~$\mu$, \ie\ $F_{\mu}(x)=\mu\bigl((-\infty,x)\bigr)$ ($x\in\R$), is equal to the distribution function of a~measure~$\J_h\nu$, where $\nu$ is given by~\eqref{prop2:4}. Then these measures are equal (\cf\ \eg~\cite[Sections~12,~14]{Bil95}), which finishes the proof of~\ref{prop21}.
  \item\medskip
   Let $\mu$ has the form~\eqref{prop2:3} with $\nu\in\mr$. Then
   \begin{equation}\label{prop2:7}
    \mu=\J_h\nu=\sum_{n=0}^{\infty}\tau_h^n\nu=\nu+\sum_{n=1}^{\infty}\tau_h^n\nu\,.
   \end{equation}
   Using the definition of $\tau_h$ and $\tau_h^n$ we get
   \begin{align*}
    \tau_h(\J_h\nu)(B)&=\J_h\nu(B-h)=\sum_{n=0}^{\infty}\tau_h^n\nu(B-h)\\
    &=\sum_{n=0}^{\infty}\tau_h(\tau_h^n\nu)(B)
     =\sum_{n=0}^{\infty}\tau_h^{n+1}\nu(B)=\sum_{n=1}^{\infty}\tau_h^n\nu(B)\,.
   \end{align*}
   Therefore $\tau_h\,\mu=\sum_{n=1}^{\infty}\tau_h^n\nu$, which, together with~\eqref{prop2:7}, yields $\mu=\nu+\tau_h\,\mu$, which implies $\nu=\mu-\tau_h\,\mu$ and the proof of~\ref{prop22}.~is finished.\hfill\null\qed
 \end{enumerate}
\end{proof}

For $\nu\in\mr$ and $h_1,\dots,h_n>0$ denote $\J_{h_1h_2\dots\,h_n}\nu=\J_{h_1}\J_{h_2}\dots\J_{h_n}\nu$. As the immediate consequence of Proposition~\ref{prop2} we obtain
\begin{prop}\label{prop3}
 Let $h_1,\dots,h_n>0$.
 \begin{enumerate}[\upshape (a)]
  \item\label{prop3a}
   If $\nu\in\mr$ fulfils the condition $\J_{h_1\dots\,h_n}\nu\in\mr$, then $\nabla_{h_1\dots\,h_n}\bigl(\,\J_{h_1\dots\,h_n}\nu\bigr)=\nu$.
  \item\smallskip
   If $\mu\in\mr$ fulfils the condition $\nabla_{h_1\dots\,h_n}\,\mu\ge 0$, then $\J_{h_1\dots\,h_n}\bigl(\nabla_{h_1\dots\,h_n}\,\mu\bigr)=\mu$.
 \end{enumerate}
\end{prop}

\subsection{Preparation to the proof of Theorem~\ref{th:main}}

Fix $n\in\N$ and consider the Hamel basis $H\subset\R$ \st\ $h_1,\dots,h_{n+1}\in H$ are distinct and positive. We keep this convention throughout the whole section. Recall that if $x\in\R$, then $\delta_x$ denotes the~Dirac measure, \ie\ $\delta_x(B)=1$ if $x\in B$ and $\delta_x(B)=0$, $x\not\in B$, where $B\subset\R$. Define the measures $\mu_1,\dots,\mu_{n+1}\in\mr$ by
\begin{equation}\label{eq:9}
 \mu_i=\J_{h_1\dots\, h_{n+1}}\delta_{h_i}\,,\quad i=1,\dots,n+1\,.
\end{equation}
Then define the signed measure $\mu$ by
\begin{equation}\label{eq:10}
 \mu=\mu_2+\dots+\mu_{n+1}-\mu_1\,.
\end{equation}
Being the elements of the Hamel basis, $h_1,\dots,h_{n+1}$ are incommensurable, and it is not difficult to check the formula
\begin{equation}\label{eq:11}
 \mu_i=\sum_{j_1,\dots,j_{n+1}=0}^{\infty}\delta_{h_i+j_1h_1+\dots+j_{n+1}h_{n+1}}\,,\quad i=1,\dots,n+1\,.
\end{equation}

Next take the sets $A,A_1,\dots,A_{n+1}\subset\R$ defined by
\begin{align*}
 A_i&=\Biggl\{h_i+\sum_{\substack{j=1\\j\ne i}}^{n+1}\varepsilon_jh_j\;:\;\varepsilon_j\in\{0,1\},\;j=1,\dots,n+1\Biggr\}\,,
 \quad i=1,\dots,n+1\,,\\
 A&=A_1\cup\dots\cup A_{n+1}\,.
\end{align*}

We will use the frequent notation $\mu(x)=\mu\bigl(\{x\}\bigr)$.
\begin{lem}\label{lem4}
 Let $i\in\{1,\dots,n+1\}$. Then
 \begin{enumerate}[\upshape (a)]
  \item\label{lem4a}
   $\mu_i(x)=1$ for $x\in A_i$,
  \item\smallskip\label{lem4b}
   $\mu_i(x)=0$ for $x\in A\setminus A_i$,
  \item\smallskip
   $\mu|_A(x)<0\iff x=h_1$,
  \item\smallskip\label{lem4d}
   $\mu(h_1)=\mu_1(h_1)=-1$,
  \item\smallskip\label{lem4e}
   $\mu_+|_A=\mu|_A+\delta_{h_1}$.
 \end{enumerate}
\end{lem}
\begin{proof}
 It is enough to use \eqref{eq:9}, \eqref{eq:10}, \eqref{eq:11}. We omit a~standard and easy proof.\hfill\null\qed
\end{proof}

Recall that in Theorem~\ref{th:main} we defined the function $f(x)=\bigl(a(x)\bigr)_+^n$ ($x\in\R$), where $a:\R\to\R$ is the additive function \st\ $a(h_1)=-1$, $a(h_2)=\dots=a(h_{n+1})=1$. Now we prove the crucial property of this function~$f$. Let us notice that the function~$f$ could be, of course, defined for any $n\in\N$ and the result below is not dependent on evenness of~$n$.
\begin{thm}\label{th5}
 Let $f:\R\to\R$ be defined as above and $\mu$ be a~signed measure given by~\eqref{eq:10}. Then
 \begin{equation}\label{eq:12}
  f(x)=(\,\mu+\delta_{h_1})^n(x)\quad\text{ for every }x\in A\,.
 \end{equation}
 In particular,
 \begin{equation}\label{eq:13}
  \nabla_{h_1\dots\,h_{n+1}}\,f(h_1+\dots+h_{n+1})=\nabla_{h_1\dots\,h_{n+1}}(\,\mu+\delta_{h_1})^n(h_1+\dots+h_{n+1})\,.
 \end{equation}
\end{thm}
\begin{proof}
 To prove~\eqref{eq:12} it is enough to show that
 \begin{equation}\label{eq:14}
  a(x)=\mu(x)\quad\text{ for every }x\in A\,.
 \end{equation}
 Indeed, then for any $x\in A$ we have $a_+(x)=\mu_+(x)$ and trivially $f(x)=\bigl(a_+(x)\bigr)^n=\bigl(\,\mu_+(x)\bigr)^n$. Taking into account Lemma~\ref{lem4}~\itemref{lem4e} we get~\eqref{eq:12}.
 \par\medskip
 To prove~\eqref{eq:14} fix $x\in A$. Then $x=\varepsilon_1h_1+\dots+\varepsilon_{n+1}h_{n+1}$, where $\varepsilon_i\in\{0,1\}$, $i=1,\dots,n+1$ and $\varepsilon_1+\dots+\varepsilon_{n+1}>0$. Two cases are possible.
 \par\medskip\noindent
 \textsc{C{a}se} 1. $\varepsilon_1=1$
 \par\medskip
  If $\varepsilon_2=\dots=\varepsilon_{n+1}=0$, then $a(x)=a(h_1)=-1$ and by Lemma~\ref{lem4}~\itemref{lem4d} $\mu(x)=\mu(h_1)=-1$, so~\eqref{eq:14} holds. If $\varepsilon_j\ne 0$ for some $j\in\{2,\dots,n+1\}$, then without loss of generality we may assume that $x=h_1+\dots+h_k$ for some $k\in\{2,\dots,n+1\}$. Since $x\in A_1\cap\dots\cap A_k$ and $x\not\in A_{k+1},\dots,x\not\in A_{n+1}$, we have by Lemma~\ref{lem4}~\itemref{lem4a},~\itemref{lem4b}
  \[
   \mu_1(x)=\dots=\mu_k(x)=1\,,\quad\mu_{k+1}(x)=\dots=\mu_{n+1}(x)=0.
  \]
  Hence, by virtue of~\eqref{eq:10}, $\mu(x)=k-2$. By additivity
  \[
   a(x)=a(h_1+\dots+h_k)=a(h_1)+\dots+a(h_k)=k-2\,,
  \]
  which proves that $a(x)=\mu(x)$.
  \par\medskip\noindent
 \textsc{C{a}se} 2. $\varepsilon_1=0$
 \par\medskip
 Without loss of generality we may assume that $x=h_2+\dots+h_k$ for some $k\in\{2,\dots,n+1\}$. Arguing exactly in the same way as in the previous case, we arrive at $\mu(x)=k-1=a(x)$.\hfill\null\qed
\end{proof}

\subsection{Proof of Theorem~\ref{th:main}}

Theorem~\ref{th5} allows us to work with measures instead of the original function~$f$. We present below three useful formulas. We will prove them after the proof of Theorem~\ref{th:main}.
\begin{lem}\label{lem_useful}
 Let $\mu=\mu_2+\dots+\mu_{n+1}-\mu_1$ be the signed measure given by~\eqref{eq:10}. Then
 \begin{align}
  &(\,\mu+\delta_{h_1})^n(x)=\mu^n(x)-(-1)^n\delta_{h_1}(x)\quad\text{for any }x\in A\,,\label{lem7}\\
  &\nabla_{h_1\dots\,h_{n+1}}\,\mu^n(h_1+\dots+h_{n+1})=0\,,\label{lem8}\\
  &\nabla_{h_1\dots\,h_{n+1}}\delta_{h_1}(h_1+\dots+h_{n+1})=(-1)^n\,.\label{lem10}
 \end{align}
\end{lem}
\begin{proof_main}
 Recall that $n\in\N$ was odd and we have chosen the Hamel basis $H\subset\R$ \st\ $h_1,\dots,h_{n+1}\in H$ were positive. We took the additive function $a:\R\to\R$ \st\ $a(h_1)=-1$ and $a(h_2)=\dots=a(h_{n+1})=1$. Then we defined the function $f:\R\to\R$ by $f(x)=\bigl(a(x)\bigr)_+^n$ and we have shown that $f$~is $n$--Jensen--convex. It was left to prove that $f$~is not $n$--Wright--convex. To show it it is enough to check that $\Delta_{h_1\dots\,h_{n+1}}\,f(0)=-1$. By~\eqref{eq:nabla_n} it is equivalent to $\nabla_{h_1\dots\,h_{n+1}}\,f(h_1+\dots+h_{n+1})=-1$. Using~\eqref{eq:13} and~\eqref{lem7} we obtain
 \begin{multline*}
  \nabla_{h_1\dots\,h_{n+1}}\,f(h_1+\dots+h_{n+1})=\nabla_{h_1\dots\,h_{n+1}}(\,\mu+\delta_{h_1})^n(h_1+\dots+h_{n+1})\\
  =\nabla_{h_1\dots\,h_{n+1}}\Bigl(\mu^n(h_1+\dots+h_{n+1})-(-1)^n\delta_{h_1}(h_1+\dots+h_{n+1})\Bigr)\\
  =\nabla_{h_1\dots\,h_{n+1}}\,\mu^n(h_1+\dots+h_{n+1})+\nabla_{h_1\dots\,h_{n+1}}\delta_{h_1}(h_1+\dots+h_{n+1})=-1
 \end{multline*}
 due to~\eqref{lem8} and~\eqref{lem10}. This finishes the proof.\hfill\null\qed
\end{proof_main}

\subsection{Proof of Lemma~\ref{lem_useful}}

\begin{proof7}
 Let $x\in A$. Of course $\delta_{h_1}^j=\delta_{h_1}$ ($j\in\N$). Therefore
 \begin{equation}\label{eq:16}
  \begin{aligned}
   (\,\mu+\delta_{h_1})^n(x)&=\mu^n(x)+\delta_{h_1}^n(x)
                             +\sum_{k=1}^{n-1}\binom{n}{k}\,\mu^k(x)\,\delta_{h_1}^{n-k}(x)\\
                            &=\mu^n(x)+\delta_{h_1}(x)
                             +\sum_{k=1}^{n-1}\binom{n}{k}\,\mu^k(x)\,\delta_{h_1}(x)\,.
  \end{aligned}
 \end{equation}
 Put $\lambda=\mu_2+\dots+\mu_{n+1}$. Then $\mu=\lambda-\mu_1$ and for $k=1,\dots,n-1$ we get
 \begin{equation}\label{eq:17}
  \mu^k(x)=\sum_{j=1}^{k-1}\binom{k}{j}\,\lambda^j(x)\Bigl(-\mu_1^{k-j}(x)\Bigr)
  +\lambda^k(x)+(-1)^k\mu_1^k(x)\,.
 \end{equation}
 It is easy to see that
 \begin{enumerate}[a)]
  \item
   $\lambda^j(x)\,\delta_{h_1}(x)=0$, $j=1,\dots,k$,
  \item
   $\mu_1^k(x)\,\delta_{h_1}(x)=\delta_{h_1}(x)$.
 \end{enumerate}
 For, notice that $\lambda$ is concentrated on the set $A_2\cap\dots\cap A_{n+1}$ (see~\ref{lem4}~\itemref{lem4a},~\itemref{lem4b}), which gives~a), while~b) is trivial. Then~\eqref{eq:17} yields
 \[
  \mu^k(x)\,\delta_{h_1}(x)=(-1)^k\delta_{h_1}(x)\,,\quad k=1,\dots,n-1\,,
 \]
 and, consequently,
 \begin{align*}
  \sum_{k=1}^{n-1}\binom{n}{k}\,\mu^k(x)\,\delta_{h_1}(x)&=
   \sum_{k=1}^{n-1}\binom{n}{k}(-1)^k\delta_{h_1}(x)
   =\delta_{h_1}(x)\Biggl[\,\sum_{k=1}^{n-1}\binom{n}{k}(-1)^k\Biggr]\\
   &=\delta_{h_1}(x)\Biggl[\,\sum_{k=0}^n\binom{n}{k}(-1)^k-\binom{n}{0}(-1)^0-\binom{n}{n}(-1)^n\Biggr]\\
   &=\delta_{h_1}(x)\bigl[\,0-1-(-1)^n\bigr]\,.
 \end{align*}
 We conclude the proof putting this last equation into~\eqref{eq:16}.\hfill\null\qed
\end{proof7}

\begin{proof8}
 Let $x\in A$. Applying the Multinomial Theorem to $\mu=\mu_2+\dots+\mu_{n+1}-\mu_1$ we arrive at
 \begin{equation}\label{eq:20}
  \mu^n(x)=\sum_{j_1+\dots+j_{n+1}=n}\binom{n}{j_1,j_2,\dots,j_{n+1}}\,
  \mu_2^{j_2}(x)\cdot\dotsc\cdot\mu_{n+1}^{j_{n+1}}(x)\cdot\Bigl(-\mu_1^{j_1}(x)\Bigr)\,,
 \end{equation}
 where
 \[
  \binom{n}{j_1,j_2,\dots,j_{n+1}}=\frac{n!}{j_1!\cdot j_2!\cdot\dotsc\cdot j_{n+1}!}
 \]
 are the \emph{multinomial coefficients}. Due to Lemma~\ref{lem4}~\itemref{lem4a},~\itemref{lem4b} we have
 \begin{equation}\label{eq:21}
  \mu_j^k(x)=
  \begin{cases}
   \mu_j(x)&\text{for }k=1,2,\dots\,,\\
   1&\text{for }k=0
  \end{cases}
 \end{equation}
 with the convention $0^0=1$. Next we will prove that
 \begin{equation}\label{eq:22}
  \mu_{j_1}(x)\cdot\mu_{j_2}(x)\cdot\dotsc\cdot\mu_{j_k}(x)=\J_{h_1\dots\,h_{n+1}}\delta_{h_{j_1}+\dots+h_{j_k}}(x)\,.
 \end{equation}
 For simplicity we will only check that
 \begin{equation}\label{eq:22a}
  \mu_1(x)\cdot\mu_2(x)=\J_{h_1\dots\,h_{n+1}}\delta_{h_1+h_2}(x)\,,
 \end{equation}
 the proof in the general case is analogous. By~\eqref{eq:11}
 \begin{align*}
  \mu_1(x)&=\sum_{j_1,\dots,j_{n+1}=0}^{\infty}\delta_{h_1+j_1h_1+\dots+j_{n+1}h_{n+1}}(x)\,,\\
  \mu_2(x)&=\sum_{j_1,\dots,j_{n+1}=0}^{\infty}\delta_{h_2+j_1h_1+\dots+j_{n+1}h_{n+1}}(x)\,.\\
 \end{align*}
 If $x\in A_1\cap A_2$, then $x=h_1+h_2+j_1h_1+\dots+j_{n+1}h_{n+1}$. Because $h_1,\dots,h_{n+1}$ belong to the Hamel basis, this representation is unique and both the above sums are equal to~$1$. By the same argument, also
 \[
  \J_{h_1\dots\,h_{n+1}}\delta_{h_1+h_2}(x)=\sum_{j_1,\dots,j_{n+1}=0}^{\infty}\delta_{h_1+h_2+j_1h_1+\dots+j_{n+1}h_{n+1}}(x)=1
 \]
 By Lemma~\ref{lem4}~\eqref{lem4a} we infer that $\mu_1(x)\mu_2(x)=1$ and~\eqref{eq:22a} holds. The remaining case $x\in A\setminus (A_1\cup A_2)$ we handle in the similar way, using also Lemma~\ref{lem4}~\eqref{lem4b}.
 \par\medskip
 Taking into account~\eqref{eq:21} and~\eqref{eq:22} we obtain that
 \begin{equation}\label{eq:23}
  \mu_2^{j_2}(x)\cdot\dotsc\cdot\mu_{n+1}^{j_{n+1}}(x)\cdot\Bigl(-\mu_1(x)^{j_1}\Bigr)
  =(-1)^{j_1}\J_{h_1\dots\,h_{n+1}}\delta_{\varepsilon_2h_2+\dots+\varepsilon_{n+1}h_{n+1}+\varepsilon_1h_1}(x)\,,
 \end{equation}
 where
 \begin{equation}\label{eq:23a}
  \varepsilon_k=
  \begin{cases}
   0&\text{for }j_k=0\,,\\
   1&\text{for }j_k>0
  \end{cases}
 \end{equation}
 for $k=1,\dots,n+1$.
 By Proposition~\ref{prop3}~\itemref{prop3a} we have
 \begin{equation}\label{eq:24}
  \nabla_{h_1\dots\,h_{n+1}}\bigl(\,\J_{h_1\dots\,h_{n+1}}
   \delta_{\varepsilon_2h_2+\dots+\varepsilon_{n+1}h_{n+1}+\varepsilon_1h_1}\bigr)(x)
   =\delta_{\varepsilon_2h_2+\dots+\varepsilon_{n+1}h_{n+1}+\varepsilon_1h_1}(x)\,.
 \end{equation}
 Consequently, by \eqref{eq:20}, \eqref{eq:23} and \eqref{eq:24} we get
 \begin{equation}\label{eq:25}
  \nabla_{h_1\dots\,h_{n+1}}\,\mu^n(x)=
  \sum_{j_1+\dots+j_{n+1}=n}\binom{n}{j_1,j_2,\dots,j_{n+1}}\,
  (-1)^{j_1}\delta_{\varepsilon_1h_1+\dots+\varepsilon_{n+1}h_{n+1}}(x)
 \end{equation}
 Observe that for $x=h_1+\dots+h_{n+1}$ there is $\varepsilon_1=\dots=\varepsilon_{n+1}=1$, so, by~\eqref{eq:23a}, $j_1+\dots+j_{n+1}\ge n+1$ and in the sum~\eqref{eq:25} there is no the component $\delta_{h_1+\dots+h_{n+1}}(h_1+\dots+h_{n+1})$. Because $h_1,\dots,h_{n+1}$ belong to the Hamel basis, every component of this sum equals~$0$, so $\nabla_{h_1\dots\,h_{n+1}}\,\mu^n(h_1+\dots+h_{n+1})=0$ and the formula~\eqref{lem8} is true.\hfill\null\qed
\end{proof8}

\begin{proof10}
 By~\eqref{eq:nabla_n} we have
 \[
  \nabla_{h_1\dots\,h_{n+1}}\,\delta_{h_1}(h_1+\dots+h_{n+1})=\Delta_{h_1\dots\,h_{n+1}}\,\delta_{h_1}(0)\,.
 \]
We compute this term using~\eqref{eq:delta_n}. Notice that (by the choice of $h_1,\dots,h_{n+1}$ as distinct elements of the Hamel basis) the only non--zero component of the sum occurring there is $(-1)^n\delta_{h_1}(h_1)=(-1)^n$, which appears in the penultimate line.\hfill\null\qed
\end{proof10}


\end{document}